\documentclass[11pt]{article}
\usepackage[top=1in, bottom=1in, left=1in, right=1in]{geometry}

\usepackage{euscript}
\usepackage{microtype,enumitem}
\usepackage{graphicx}
\graphicspath{{figs/}}

\usepackage{subfigure,algorithm,algorithmic}
\usepackage{booktabs} %
\usepackage[usenames,svgnames,dvipsnames]{xcolor}
\usepackage[linktocpage,colorlinks,linkcolor=blue,anchorcolor=blue,citecolor=blue,urlcolor=blue,pagebackref]{hyperref}
\usepackage{natbib}
\setcitestyle{authoryear,open={(},close={)}} 

\usepackage{amsmath}
\usepackage{amssymb}
\usepackage{mathtools}
\usepackage{amsthm}
\usepackage[textsize=tiny]{todonotes}
\usepackage{mathrsfs}
\usepackage{nicefrac}

\usepackage[capitalize,noabbrev%
]{cleveref}

\usepackage{pgfplots}
\DeclareUnicodeCharacter{2212}{−}
\usepgfplotslibrary{groupplots,dateplot}
\usetikzlibrary{patterns,shapes.arrows,external}

\newcommand{\figdir}{-figs/}
\pgfplotsset{compat=newest,
    width=12cm,
    height=7cm,
    every axis plot/.append style={thick},
}

\theoremstyle{plain}
\newtheorem{theorem}{Theorem}[section]
\newtheorem{proposition}[theorem]{Proposition}
\newtheorem{lemma}[theorem]{Lemma}

\theoremstyle{definition}

\newtheorem{assumption}[theorem]{Assumption}
\theoremstyle{remark}
\newtheorem{remark}[theorem]{Remark}
\newtheorem{example}{Example}[section] 

\renewcommand*{\backref}[1]{\ifx#1\relax \else Page #1 \fi}
\renewcommand*{\backrefalt}[4]{%
  \ifcase #1 \footnotesize{(Not cited.)}%
  \or        \footnotesize{(Cited on page~#2.)}%
  \else      \footnotesize{(Cited on pages~#2.)}%
  \fi
}

\usepackage{macros}

\usepackage[nolist]{acronym}
\begin{acronym}
    \acro{MCMC}{Markov Chain Monte Carlo}
    \acro{KL}{Kullback--Leibler}
\end{acronym}

\newcommand{\mbXY}{{\mathbf {XY}}}
\newcommand{\rd}{\mathrm{d}}
 
\newcommand{\ide}{\mathbf{I}}
\def\Rd{{\mathbb{R}^d}}

\allowdisplaybreaks{}

\crefname{Inequality}{Inequality}{Inequalities}

\author{Bo Yuan\thanks{Georgia Institute of Technology. \texttt{byuan48@gatech.edu} }
\and Jiaojiao Fan \thanks{Georgia Institute of Technology. \texttt{jiaojiaofan@gatech.edu}} 
\and Jiaming Liang\thanks{Yale University. \texttt{jiaming.liang@yale.edu}}
\and Andre Wibisono\thanks{Yale University. \texttt{andre.wibisono@yale.edu} }
\and Yongxin Chen\thanks{Georgia Institute of Technology. \texttt{yongchen@gatech.edu}}
}

\begin{document}

\title{On a Class of Gibbs Sampling over Networks}
\maketitle

\begin{abstract}
    We consider the sampling problem from a composite distribution whose potential (negative log density) is $\sum_{i=1}^n f_i(x_i)+\sum_{j=1}^m g_j(y_j)+\sum_{i=1}^n\sum_{j=1}^m\nicefrac{\sigma_{ij}}{2\eta} \Vert x_i-y_j \Vert^2_2$ where each of $x_i$ and $y_j$ is in $\Rd$, $f_1, f_2, \ldots, f_n, g_1, g_2, \ldots, g_m$ are strongly convex functions, and $\{\sigma_{ij}\}$ encodes a network structure.
    Building on the Gibbs sampling method, we develop an efficient sampling framework for this problem when the network is a bipartite graph. More importantly, we establish a non-asymptotic linear convergence rate for it. This work extends earlier works that involve only a graph with two nodes \cite{lee2021structured}. To the best of our knowledge, our result represents the first non-asymptotic analysis of a Gibbs sampler for structured log-concave distributions over networks.
    Our framework can be potentially used to sample from the distribution 
    $ \propto \exp(-\sum_{i=1}^n f_i(x)-\sum_{j=1}^m g_j(x))$ in a distributed manner. 
\end{abstract}

\section{Introduction}\label{section:introduction}

Sampling has been increasingly important in the areas of  computer science and Bayesian inference as it is often the critical element for parameter estimations, simulations, and numerical approximations. Designing and analysis of sampling for log-concave distributions whose negative log density is convex have become an active research field since the randomized convex body volume approximation algorithm proposed by ~\citet{dyer1991random}. Depending on various structured distributions including non-convex potentials and associated efficient sampling algorithms, recently a large number of non-asymptotic results are established under different scenarios, \textit{e.g.},  Langevin Monto Carlo ~\citep{dalalyan2017theoretical,cheng2018underdamped,dalalyan2018sampling,li2021sqrt,zhang2023improved}, proximal samplers~\citep{lee2021structured, CheCheSalWib22,gopi2022private,fan2023improved,altschuler2023faster}, lower bounds of convergence rates~\citep{chewi2022query,chewi2023fisher,chewi2023query}, etc.

In this work, we establish a linear convergence rate of a Gibbs sampler for the following unnormalized target distribution $\propto$
\begin{equation}\label{eq:target}
    \exp\left(-\sum_{i=1}^n f_i(x_i)-\sum_{j=1}^m g_j(y_j)-\sum_{i=1}^n\sum_{j=1}^m\frac{\sigma_{ij}}{2\eta} \Vert x_i-y_j \Vert^2_2\right).
\end{equation}
Here, both $f_i$ and $g_j$ functions are strongly convex (see Section \ref{subsection:background} for
definitions), and $x_i$ and $y_j$ represent random variables on vertices on the underlying network. Our formalization indicates the network considered in this work is a bipartite graph where the two disjoint sets are expressed as $\{x_i\}_{i=1}^n$ and $\{y_j\}_{j=1}^m$, respectively. In the expression, the parameter $\sigma_{ij}$ encodes the graph structure, \textit{i.e.}, $0<\sigma_{ij}\le 1$ if there is an edge between vertex $x_i$ and vertex $y_j$. Otherwise, $\sigma_{ij}=0$. We also utilize a positive coefficient $\eta$ to control the dependency between variables.

The structured sampling problem \eqref{eq:target} has potential applications in the fields of graphical models, distributed sampling, federated learning, etc. For instance, the joint distribution over Hidden Markov models or spanning trees with edge potentials being Gaussian distributions can be formulated as our target distribution, where $\nicefrac{\eta}{\sigma_{ij}}\ide$ is the covariance matrix of edge potential for the edge between vertices $x_i$ and $y_j$. Our analysis can be straightforwardly generalized to the case that the covariance matrix is not diagonal. Other
applications include estimation problems for robotics where the quadratic terms correspond to
odometry measurements and the other terms $f_i,g_j$ correspond to other measurements, \textit{e.g.}, perception.

In proximal samplers, for a target distribution $\exp(-h(x))$, a Gibbs sampling framework is utilized to sample from the joint distribution $\exp(-h(x)-\nicefrac{1}{2\eta}\Vert x-y \Vert^2_2)$. Since the $X$-marginal distribution of the joint one is exactly $\exp(-h(x))$, it is sufficient to generate samples from the joint one efficiently. In brief, as the conditional distribution $\exp(-h(x)-\nicefrac{1}{2\eta}\Vert x-y_0 \Vert^2_2)$ for any given $y_0$ is a better-conditioned distribution than the target, the proximal sampler can be proved to have  stronger convergence rate under various assumptions on $h(x)$ and implementations of sampling from the conditional measure~\citep{CheCheSalWib22}. To improve the performance of proximal samplers, the key is to  efficiently implement the sampler for conditional distributions, namely, Restricted Gaussian Oracle (RGO) ~\citep{lee2021structured}. Among recent works in this direction, \citet{gopi2022private,fan2023improved} utilized an approximate Rejection sampling framework and~\citet{altschuler2023faster} adopted a well-designed warm start. The proximal samplers achieve state-of-the-art results in a variety of settings of sampling~\citep{fan2023improved,altschuler2023faster}. Our target is a generalization of both proximal samplers and the distributed variable splitting MCMC~\citep{vono2022efficient} for large-scale Bayesian inference problems. Similarly, in our case, if $n=m=1$, our target \eqref{eq:target} given a  sufficiently small $\eta$ approximates a lower-dimensional composite distribution, $\exp(-f(x)-g(x))$. More precisely, we have shown that the distribution of samples converges in terms of total variation distance (see Section \ref{section:small_eta}  for details). 

Due to the structure of our target distribution \eqref{eq:target}, we adopt Gibbs Sampling in this work as it can  be potentially used to realize a fully distributed algorithm for sampling as in~\citet{vono2022efficient}.
In addition to the distributed nature and promising theoretical properties of Gibbs samplers, another reason to consider Gibbs samplers is the widely-used applications~\citep{daphne2009probabilistic}. Analyzing its non-asymptotic behaviors is of great importance but challenging. A pioneering work that establishes the quantitative convergence rate for general Gibbs samplers was proposed by \citet{Ros95} via the drift and minorization conditions. A long line of works was proposed for the mixing time of Gibbs samplers over graphs.~\citep{hrycej1990gibbs,venugopal2012lifting,de2015rapidly,zhang2019poisson, vono2022efficient}. In recent works, ~\citet{zhang2019poisson} analyzed the measured of a mini-batched Gibbs sampler measured by the spectral gap~\citep{de2018minibatch}, which has been used for analyzing the plain Gibbs sampling~\citep{levin2017markov}. A recently developed class of samplers for large-scale graphical models is distributed variable splitting MCMC. Among them, sampling is performed on an artificial hierarchical Bayesian model akin to the alternating direction
method of multipliers (ADMM) in optimization~\citep{vono2019split,rendell2020global,vono2020asymptotically,vono2022efficient}. To sample from $\exp(-\sum_{i=1}^n h_i(x))$, ~\citet{vono2022efficient} considers an augmented distribution $\exp(-\sum_{i=1}^n h_i(x_i)-\nicefrac{1}{2\eta}\|x_i-x\|^2)$ and shows the augmented one approximates the target by a sufficiently small $\eta$. Hence it is sufficient to sample from the augmented one with a Gibbs sampler.  In ~\citet{vono2022efficient}, the mixing time measured in 1-Wasserstein and TV distances have been given by Ricci curvature and coupling techniques. 
Inspired by the recent progress in proximal sampling, we give the first non-asymptotic analysis of a Gibbs sampler for \eqref{eq:target} by a new strong data-processing inequality on diffusion processes. Moreover, in ~\citet{vono2022efficient}, only the exact implementation of RGO was studied. In contrast, in Section \ref{section:twonodes}, we discuss the overall mixing time of our algorithm with the inexact implementations of RGO in ~\citet{fan2023improved}.

Section \ref{section:preliminaries} presents basic definitions and notation used throughout the paper.
In Section \ref{section:twonodes}, we formally prove the linear convergence rate of a Gibbs sampler for a two-node graph with a newly-established strong data processing inequality \citep{cover1999elements}. In Section \ref{section:main}, we generalize the result to distributions over the general bipartite graph. Section \ref{section:small_eta} gives the analysis of convergence to the composite distribution $\exp(-f(x)-g(x))$, where we prove the asymptotic convergence and give a non-asymptotic convergence rate for special cases. 
Section \ref{sec:discussion} presents some concluding remarks and possible extension.
Appendices \ref{section:appendix}-\ref{appendix:small_eta} provide a few useful technical results and missing proofs in the paper.

\section{Preliminaries}\label{section:preliminaries}
\subsection{Notations}
A Gaussian distribution is denoted by $\mN$ and the notation $\ide$ represents an identity matrix whose dimension is clear from the context. Assume $\pi^{XY}$ denote the joint distribution of $(X,Y)$, then the $X$-marginal distribution is represented by $\pi^X$, and the $X$-conditional distribution given $y$ is $\pi^{X|Y=y}$. Moreover, the joint distribution of multivariable $\mbX$ and $\mbY$ is $\pi^{\mbXY}$, and similarly, its $\mbX$-marginal distribution and $\mbX$-conditional distribution are denoted by $\pi^{\mbX}$ and $\pi^{\mbX|\mbY}$, respectively.

\subsection{Sampling Background}\label{subsection:background}

We start by introducing strong convexity and smoothness, which are the main assumptions in this paper. In what follows, we use $\nabla f$ to represent the subgradient of $f$. A function $f: \Rd \rightarrow \bbR$ is $\alpha$-strongly convex if and only if $ f(x)-f(y) \geq \nabla f(y)^T(x-y) + \nicefrac{\alpha}{2} \Vert x-y\Vert^2_2$ holds for all $x,y \in \bbR^d$. Furthermore, if the strongly convex function $f$ is twice-differentiable, then  $\alpha\ide \preceq\nabla^2 f$. A function $f: \Rd \rightarrow \bbR$ is $\beta$-smooth if and only if $f(x)-f(y) \leq \nabla f(y)^T(x-y) + \nicefrac{\beta}{2} \Vert x-y\Vert^2_2$ is true for all $x,y \in \bbR^d$. Similarly, for a twice-differentiable $\beta$-smooth function $f$, one has $ \nabla^2 f \preceq\beta\ide$.

Next we introduce the functional inequalities used in this work. A probability distribution $\nu$ satisfies logarithmic Sobolev inequality with a positive constant $\rho$ (in short LSI($\rho$)) if, for any probability distribution $\mu$ such than $\mu \ll \nu$, the \ac{KL} divergence $D_{\mathrm{KL}}(\mu || \nu)= \int \log\frac{\mu}{\nu}\mu$ and Fisher information $I(\mu || \nu)= \int \|\nabla \log \frac{\mu}{\nu}\|^2\mu$ satisfy $ D_{\mathrm{KL}}(\mu || \nu) \leq \frac{1}{2\rho} I(\mu || \nu)$.
A well-known result that connects strongly convexity and logarithmic Sobolev inequality is that if $\mu$ is a strongly log-concave density with a constant $\alpha$, then $\mu$ satisfies LSI($\alpha$). Moreover, Pinsker's inequality is $\mathrm{TV}(\mu,\nu) \le \sqrt{\frac{1}{2}D_{\mathrm{KL}}(\mu || \nu)}$.
Here the total variation between $\mu$ and $\nu$ satisfies $\mathrm{TV}(\mu,\nu) = \nicefrac{1}{2}\|\mu-\nu\|_1$. In addtion, we say a distribution $\nu$ satisfies Poincar\'{e} inequality with a constant $\rho>0$ (in short PI($\rho$)) if for any functions $f : \bbR^d \rightarrow \bbR$ with $\nabla f \in L^2(\nu), \Var_\nu(f) \le \nicefrac{1}{\rho}\bbE_{\nu}(\|\nabla f\|^2_2)$. It is well-known that LSI($\rho$) implies PI($\rho$).

\section{Improved analysis of Composite Sampling}
\label{section:twonodes}
\begin{assumption}\label{assum:assum1}
 Let $f(x)$ be an $\alpha_f$-strongly convex function and $g(y)$ be an $\alpha_g$-strongly convex function.
\end{assumption}
We are interested in the sampling problem from target distribution \eqref{eq:target}. For the sake of presentation, we start presenting the analysis with the two-node case as the proof of the bipartite graph is essentially the same as this one. In this case, the target distribution 
\begin{equation}\label{eq:target_twonodes}
    \pi^{XY} \propto \exp\left(-f(x)-g(y)-\frac{1}{2\eta}\Vert x-y\Vert^2_2\right)
\end{equation} under Assumption \ref{assum:assum1}. 
Based on the Gibbs sampling framework, our sampling algorithm \ref{algo:algo1} runs as follows. 

\begin{algorithm}[H]
    \caption{A Gibbs sampler for a two-node graph}
    \begin{algorithmic}\label{algo:algo1}
    \STATE \textbf{Input: } $\pi^{XY}$: the target distribution 
    \STATE $x^0$: an initial sample drawn from $\mu_0^X$
    \STATE \textbf{Output: } $(x^K,y^K)$: samples generated in the $k$-th iteration
    \FOR{$k \leftarrow 0,\cdots, K$}
    \STATE Step 1: Sample $y^k \sim \pi^{Y|X=x^{k}} \propto  \exp (-g(y)-\nicefrac{1}{2\eta}\Vert x^{k}-y\Vert^2_2) $.
    \STATE Step 2: Sample $x^{k+1} \sim \pi^{X|Y=y^{k}} \propto  \exp(-f(x)-\nicefrac{1}{2\eta}\Vert x-y^{k}\Vert^2_2) $.
    \ENDFOR
\end{algorithmic}
\end{algorithm}

The critical step in Algorithm \ref{algo:algo1} is to sample from the conditional distribution, \textit{i.e.}, Step 1 and 2. To analyze these steps, we construct an auxiliary diffusion process to represent the sampling from conditional distributions in Lemma \ref{lemma:diffusion}. Note these processes are constructed merely for analysis. There exist efficient implementations of the RGO; see the discussions at the end of this section. To understand the dynamics of the diffusion process, we establish a strong data processing inequality on it by Lemma \ref{lemma:basic} and Theorem \ref{theorem:dataprocessing}. In this way, we are able to show the sample distribution converges to the target one \eqref{eq:target_twonodes} in terms of \ac{KL} divergence in Theorem \ref{theorem:twonode}.

\begin{lemma}\label{lemma:basic}
Given an arbitrary diffusion 
    \begin{equation}\nonumber
        \rd\mbZ_t = b_t (\mbZ_t) \rd t + \boldsymbol{\sigma}_t \rd\mbW_t,
    \end{equation}
where $\mbZ_t \in \bbR^d$ and $\boldsymbol{\sigma}_t$ is a $d\times d$ $\boldsymbol{diagonal}$ matrix. Furthermore denote the $i$-th entry of $\boldsymbol{\sigma}_t$  as $\sigma_{t,i}$ and assume  $\lambda_t^2 \le \min_i\sigma_{t,i}^2$. Let $\mu_t, \pi_t$ be the distributions of $\mbZ_t$ initialized with $\mbZ_0 \sim \mu_0, \mbZ_0\sim \pi_0$ respectively. Then we have 
\begin{equation}\label{eq:Lemma2}
         \frac{\partial D_{\mathrm{KL}}(\mu_t || \pi_t) }{\partial t} \le-\frac{\lambda_t^2}{2}I(\mu_t||\pi_t).
\end{equation}
\end{lemma}
\begin{proof}
The proof is based on the Fokker–Planck equation. See the full proof in Appendix \ref{subsection:appendix_A1}.
\end{proof}
Lemma \ref{lemma:basic} is different from the
classical result describing the convergence to the stationary distribution of a Langevin diffusion. In
Lemma \ref{lemma:basic}, neither $\mu_t$ nor $\pi_t$ needs to be the stationary distribution; they can both be time-varying.
To our best knowledge, only a special case of it involving simultaneous heat flow (\textit{i.e.}, $b_t(Z_t) \equiv 0$)
was recently used in a few works \citep{lee2021structured,CheCheSalWib22}.

\begin{theorem}[Strong data processing inequality with LSI]\label{theorem:dataprocessing}
    Under the condition of Lemma \ref{lemma:basic}, if we further assume $\pi_t$ satisfies the LSI with a coefficient $\alpha_t, \: \forall t\in [0, T]$, then 
    \begin{equation}\nonumber
        D_{\mathrm{KL}} (\mu_T \| \pi_T) \le \exp\left(-\int_0^T \alpha_t \lambda_t^2 \, \rd t\right) D_{\mathrm{KL}} (\mu_0 \| \pi_0).
    \end{equation}
\end{theorem}
\begin{proof}
With the logarithmic Sobolev inequality of $\pi_t$ and {\eqref{eq:Lemma2}}, we have
    \begin{equation}\nonumber
        \frac{\partial D_{\mathrm{KL}}(\mu_t || \pi_t) }{\partial t}   \leq -\alpha_t\lambda_t^2 D_{\mathrm{KL}}(\mu_t || \pi_t).
    \end{equation}
Solving the above differential inequality gives
    \begin{equation}\nonumber
         D_{\mathrm {KL}} (\mu_T \| \pi_T) \le \exp\left(-\int_0^T \alpha_t \lambda_t^2 \, \rd t\right) D_{\mathrm {KL}} (\mu_0 \| \pi_0).
    \end{equation}
\end{proof}
Note that Theorem \ref{theorem:dataprocessing} is tight if the elements in  $\boldsymbol{\sigma}_t$ is the same. This can be seen from a simple example $\boldsymbol{b}_t(\boldsymbol{Z}_t)=-\boldsymbol{Z}_t$, $\boldsymbol{\sigma}_t=\sqrt{2}\ide$, $\mu_0=\mN(1,1)$ and $\nu_0=\mN(2,1)$ in Lemma \ref{lemma:basic}.
 Lemma \ref{lemma:basic} is a generalization of Lemma 16 in \citet{vempala2019rapid} and Lemma 12 in \citet{CheCheSalWib22}. In Lemma \ref{lemma:basic}, we assume each  entry of the diagonal matrix $\boldsymbol{\sigma}_t$ can be different. It is worth mentioning that similar results can be established with other probability divergences. Under the LSI assumption, one can obtain the contraction property with R\'{e}nyi divergence (Section A.4 in \citet{CheCheSalWib22}). Moreover, similar results still hold with different assumptions for $\pi_t$. For instance, in Appendix \ref{appendix:convexity}, we show a data processing inequality with the convexity assumption. We refer the reader to \citet{CheCheSalWib22} for a comprehensive study.
Next we show how to construct the diffusion process for a given joint distribution of two endpoints. 
\begin{lemma}\label{lemma:diffusion} 
Define a distribution $P(z(\cdot))$ over the space of trajectories $z(\cdot) \in C([0, 1]; \bbR^d)$ as follows: $P(z(0), z(1)) \propto \exp(-f(z(0))-g(z(1))-\nicefrac{1}{2\eta} \|z(0)-z(1)\|^2)$ is the joint distribution of $Z_0$ and $Z_1$, and $P(z(\cdot)|z(0),z(1))$ is the process of $\rd Z_t = \sqrt{\eta} \rd W_t$ conditioning on $Z_0 = z(0), Z_1 = z(1)$ (a.k.a., the Brownian bridge). Then $P$ has the diffusion representation 
    \begin{equation}\label{eq:forward}
        \rd Z_t = \eta \nabla \log \phi_t (Z_t) \rd t + \sqrt{\eta} \rd W_t, \quad Z_0 \sim P(z(0))
    \end{equation}
in the forward direction and 
    \begin{equation}\label{eq:backward}
        \rd Z_t = -\eta \nabla \log \hat\phi_t (Z_t) \rd t + \sqrt{\eta} \rd W_t, \quad Z_1 \sim P(z(1))
    \end{equation}
in the backward direction. Here 
    \begin{align}\label{eq:diffusion1}
    \begin{split}
    \phi_t(z)&= \int \exp[-g(y)-\frac{1}{2\eta(1-t)}\Vert z-y \Vert^2_2] \rd y
    \\
    \hat \phi_t(z) &= \int \exp[-f(x)-\frac{1}{2\eta t}\Vert z-x \Vert^2_2] \rd x.
    \end{split}
    \end{align}
Moreover, the marginal distribution of $P$ satisfies $P(z(t)) \propto \phi_t(z(t)) \hat\phi_t(z(t))$.
\end{lemma}
\begin{proof}
We show that $P(Z_\cdot|Z_0, Z_1)$, the conditional distribution of \eqref{eq:forward}, coincides with the Brownian bridge. See the full proof in Appendix \ref{subsection:appendix_A2}.
\end{proof}
In Lemma \ref{lemma:diffusion}, the forward diffusion \eqref{eq:forward} only relies
on the conditional distribution $P(z(1)|z(0))$.
It is also worth pointing out that the diffusion process in Lemma \ref{lemma:diffusion} is exactly the solution of the Schr\"{o}dinger bridge problem. Meanwhile, \eqref{eq:diffusion1} is the solution of the Schr\"{o}dinger system, one way to solve the Schr\"{o}dinger bridge problem. We refer the reader to \citet{leonard2014survey,chen2016relation,chen2021stochastic,pavon2021data,vargas2021machine} for more detailed explanations. Theorem \ref{theorem:twonode} is the main result for the two-node case. With Theorem \ref{theorem:dataprocessing} and Lemma \ref{lemma:diffusion}, we have the linear convergence rate for the Gibbs sampling framework shown as follows.
\begin{theorem}[Convergence result with strong-convexity]\label{theorem:twonode} 
Denote the distribution of the $k$-th samples $(x^k,y^k)$ by $\mu^{XY}_k$. For the target distribution \eqref{eq:target_twonodes}, under Assumption \ref{assum:assum1}, 
\begin{equation}
  D_{\mathrm{KL}}(\mu^{XY}_k || \pi^{XY})  \leq  \frac{D_{\mathrm{KL}}(\mu_0^X || \pi^{X})}{({1+\eta\alpha_f})^{2k}({1+\eta\alpha_g})^{2k}}.
\end{equation}
\end{theorem}
\begin{proof}\label{proof:theroem1}
In Lemma \ref{lemma:diffusion}, the forward diffusion process \eqref{eq:forward} only depends on the the conditional distribution $P(Z(1)|Z(0))$, so we can construct the same diffusion process for both $(\pi^X,\pi^Y)$ and $(\mu_k^X,\mu_k^Y)$, i.e.,
    \begin{equation}\nonumber
        \rd Z_t = \eta \nabla \log \phi_t (Z_t) \rd t + \sqrt{\eta}\, \rd W_t.
    \end{equation}
As $\pi^{XY} \propto \exp(-f(x)-g(y)-\frac{1}{2\eta}\|x-y\|^2)$, we have the explicit expression for the distribution of $Z_t$ by \eqref{eq:diffusion1}
\begin{equation}
   P(z_t) \propto \int \exp[-g(y)-\frac{1}{2\eta(1-t)}\Vert z-y \Vert^2_2] \rd y  \int \exp[-f(x)-\frac{1}{2\eta t}\Vert z-x \Vert^2_2] \rd x. 
\end{equation}
Under Assumption \ref{assum:assum1}, $f(x)$ is $\alpha_f$-strongly convex and $g(y)$ is $\alpha_g$-strongly convex. Hence, by Theorem 3.7.b and 3.7.d in \cite{saumard2014log}, we get $ P(z_t)$ is strongly log-concave with a coefficient $\frac{\alpha_f}{1+\eta t\alpha_f}+\frac{\alpha_g}{1+\eta(1-t)\alpha_g}$. Therefore, $P(z_t)$ satisfies LSI$(\frac{\alpha_f}{1+\eta t\alpha_f}+\frac{\alpha_g}{1+\eta(1-t)\alpha_g})$. So, for Step 1 in Algorithm \ref{algo:algo1}, we can employ the strong data processing inequality in Theorem \ref{theorem:dataprocessing}, which yields 
\begin{equation}\nonumber
    D_{\mathrm{KL}}({\mu}_{k}^Y || \pi^Y) \leq D_{\mathrm{KL}}({\mu}_k^X || \pi^X )\frac{1}{(1+\eta\alpha_f)(1+\eta\alpha_g)}, \: \forall k.
\end{equation}
Similarly, for Step 2 in Algorithm \ref{algo:algo1} we can perform a symmetric analysis, which directly gives 
\begin{equation}\nonumber
    D_{\mathrm{KL}}({\mu}_{k+1}^X || \pi^X) \leq D_{\mathrm{KL}}({\mu}_k^Y || \pi^Y )\frac{1}{(1+\eta\alpha_f)(1+\eta\alpha_g)}, \: \forall k.
\end{equation}
Therefore, 
\begin{equation}\nonumber
    D_{\mathrm{KL}}({\mu}_{k}^X || \pi^X) \leq D_{\mathrm{KL}}({\mu}_0^X || \pi^X )\frac{1}{(1+\eta\alpha_f)^{2k}(1+\eta\alpha_g)^{2k}}.
\end{equation}
Lastly according to Lemma \ref{lemma:KL}
we have
\begin{equation}\nonumber
    D_{\mathrm{KL}}(\mu^{XY}_k || \pi^{XY}) = D_{\mathrm{KL}}(\mu^{X}_k || \pi^{X})  \leq D_{\mathrm{KL}}(\mu_0^X || \pi^{X}) \frac{1}{({1+\eta\alpha_f})^{2k}({1+\eta\alpha_g})^{2k}}.
\end{equation}
\end{proof}

Our result indicates that the convergence is faster with larger $\eta$, which means Gibbs sampling may have good performance for slightly correlated variables. For more discussion of the impact of correlated structure on Gibbs sampling, see e.g., \citet{roberts1997updating}. Notice that our result is still meaningful, even one of $f$ or $g$ is convex. However, if the strong convexity constants $\alpha_f$ and $\alpha_g$ both approach 0, \textit{i.e.}, $f$ and $g$ are convex instead of strongly convex, the convergence fails. Hence we discuss how to relax the strong-convexity assumption.
The strong convexity in Assumption \ref{assum:assum1} can be relaxed to convexity. Following similar techniques in section A.3 in \citet{CheCheSalWib22}, we establish a data processing inequality with the convexity assumption and thereby convergence rate of the two-node case. 
\begin{theorem}[Convergence rate with convexity]
\label{theorem:convexity}
Denote the distribution of the $k$-th samples $(x^k,y^k)$ by $\mu^{XY}_k$. Assume $f$ and $g$ are convex. For target distribution \eqref{eq:target_twonodes}, we have 
    \begin{equation}\nonumber
        D_{\mathrm{KL}}(\mu^{XY}_k || \pi^{XY}) \le \frac{W_2^2(\mu_0^X,\pi^X)}{k\eta}.
    \end{equation}
\end{theorem}
\begin{proof}
    See Appendix \ref{appendix:convexity} for the proof. 
\end{proof}
\begin{remark} [LSI assumption]\label{remark:LSIAssump}
    The Assumption \ref{assum:assum1} can be potentially relaxed to LSI conditions. Note that we only need LSI conditions in Theorem \ref{theorem:dataprocessing}. However, to establish the convergence rate, we need to compute the LSI constant for \eqref{eq:diffusion1}, which requires the LSI constant for the product of two densities that satisfy LSI. To the best of our knowledge, it is an open problem without assuming stronger conditions. We refer the reader to \citet{villani2021topics} for more discussions. 
\end{remark}


Next, we consider the mixing time of Algorithm \ref{algo:algo1} by combining Theorem \ref{theorem:twonode} and the implementation of RGO. For the exact implementation, one can obtain the sampling complexity of \eqref{eq:target_twonodes} with additional conditions in Assumption \ref{assum:assum4}.
\begin{assumption}\label{assum:assum4}
    Assume $f$ and $g$ are smooth with coefficients $\beta_f$ and $\beta_g$, respectively.
\end{assumption}
Under Assumption \ref{assum:assum1} and \ref{assum:assum4},
 one can employ the rejection sampling framework described in Section 4.2 in \citet{CheCheSalWib22} to have an \textbf{exact} implementation of Step 1 and 2. In short, for a strongly convex and smooth potential,  one can find the minimizer of the potential and then use rejection sampling with a Gaussian distribution centered at the minimizer. Then, the expected number of iterations for Step 1 and Step 2 is $\left(\frac{\beta_g+\nicefrac{1}{\eta}}{\alpha_g+\nicefrac{1}{\eta}}\right)^{\nicefrac{d}{2}}$ and $\left(\frac{\beta_f+\nicefrac{1}{\eta}}{\alpha_f+\nicefrac{1}{\eta}}\right)^{\nicefrac{d}{2}}$, respectively. (Theorem 7 in \citet{chewi2021query}). Hence, given the condition $\eta = \mathcal{O}(\frac{1}{(\beta_f+ \beta_g) d})$, sampling from conditional distributions only requires $\mathcal{O}(1)$ steps in expectation. Then with Theorem \ref{theorem:twonode} and Pinsker's inequality, 
to have $\delta$ TV distance to the target \eqref{eq:target_twonodes}, the number of Gibbs sampling iterations satisfies
\begin{equation}\nonumber
        K = \mathcal{O}\left(\frac{(\beta_f +\beta_g)d}{\alpha_f+\alpha_g}\log\frac{D_{\mathrm{KL}}(\mu_0^X || \pi^X)}{\delta^2}\right).
\end{equation}

Under the same assumptions, Assumption \ref{assum:assum1} and \ref{assum:assum4}, we can also use the \textbf{inexact} implementation in~\citet{fan2023improved} for Step 1 and 2. By Theorem 3 in ~\citet{fan2023improved}, the expected number of iterations for Step 1 and Step 2 are
$\mathcal{O}(1)$ given $\eta =\mathcal{O}(\frac{\log^{-1}(1/\epsilon_{\text{RGO}})}{(\beta_f+\beta_g)d^{1/2}}) $
where $\epsilon_{\text{RGO}}$ is the TV distance between the conditional
distribution and the distribution of samples. The error $\epsilon_{\text{RGO}}$ can accumulate at each iteration of the
Gibbs sampler but can be controlled due to the triangular inequality for total variation and the
logarithmic dependence of $\eta$ over $\epsilon_{\text{RGO}}$. Following a similar argument as in Proposition 1 in~\citet{fan2023improved}, we can establish the number of iterations is \begin{equation}\nonumber
        K = \Tilde{\mathcal{O}}\left(\frac{(\beta_f +\beta_g)d^{1/2}}{\alpha_f+\alpha_g}\right).
\end{equation}

Note that the smoothing assumption is not required in the proof of convergence rate. It is only necessary to establish the bound of the number of iterations when sampling from the conditional distributions. Therefore, Theorem \ref{theorem:twonode} is still applicable to strongly convex but non-smooth functions.

\section{Gibbs sampling over bipartite graphs}\label{section:main}
\begin{assumption}\label{assum:assum2}
     Let both $f_i$ and $g_j$ be strongly convex functions, and denote the strong convexity constants of $f_i$ and $g_j$ by $\alpha_{f_i}$ and $\alpha_{g_j}$, respectively. We also impose two regularity conditions for the graph structure, \textit{i.e.}, $\sum_{i=1}^n \sigma_{ij}>0, \: \forall j$ and $\sum_{j=1}^m \sigma_{ij}>0, \: \forall i$.
\end{assumption}
In this section, we consider sampling over bipartite graphs with a Gibbs sampler under Assumption \ref{assum:assum2}. A bipartite graph $G$ is a graph whose vertices can be divided into two disjoint sets $U$ and $V$ with the constraint that each edge in $G$ connects one vertex in $U$ and another one in $V$. Note that bipartite graphs include a large number of graph structures, \textit{e.g.}, a tree is always a bipartite graph.
The target distribution we sample from is of the form
\begin{equation}\nonumber
\pi^{\mbX\mbY} \propto \exp\left[-\sum_{i=1}^n f_i(x_i)-\sum_{j=1}^m g_j(y_j)-\sum_{i=1}^n\sum_{j=1}^m\frac{\sigma_{ij}}{2\eta} \Vert x_i-y_j \Vert^2_2\right]
\end{equation}
where each $x_i$ and $y_j$ represents a vertex in $U$ and $V$, respectively. If there is an edge between vertex $i$ in $U$ and vertex $j$ in $V$, then $0 < \sigma_{ij} \le 1$. Otherwise $\sigma_{ij} = 0$. Note that the two regularity conditions in Assumption \ref{assum:assum2} imply that for any vertex in $G$, at least one edge connects it with the rest part of $G$. Next we present a blocked Gibbs sampling algorithm over a bipartite graph $G$. In Algorithm \ref{algo:algo2}, $\mbX^k_i$ denotes the sample on vertex $i$ in the $k$-th iteration. 
\begin{algorithm}[H]
    \caption{A Gibbs sampler over a bipartite graph}
    \begin{algorithmic}\label{algo:algo2}
    \STATE \textbf{Input: } $\pi^{\mbX\mbY} \propto \exp[-\sum_{i=1}^n f_i(x_i)-\sum_{j=1}^m g_j(y_j)-\sum_{i=1}^n\sum_{j=1}^m\frac{\sigma_{ij}}{2\eta} \Vert x_i-y_j \Vert^2_2]$: the target distribution 
    \STATE $\mbX^0$: an initial sample of $\mbX$ drawn from $\mu_0^{\mbX}$
    \STATE \textbf{Output: } $({\mbX}^K,{\mbY}^K)$: samples that approximate the target distribution $\pi^{\mbX\mbY}$
    \FOR{$k \leftarrow 0,\cdots, K$}
        \FOR{$j \leftarrow 1,\cdots, m$}
        \STATE Step 1: Sample $\mbY^k_j \sim \pi^{\mbY_j|\mbX=\mbX^k}\propto  \exp \left[-g_j(y_j)-\sum_{i=1}^{n}\frac{\sigma_{ij}}{2\eta}\Vert \mbX_i^k-y_j\Vert^2_2\right] $.
        \ENDFOR
        \FOR{$i \leftarrow 1,\cdots, n$}
        \STATE Step 2: Sample $\mbX^{k+1}_i \sim \pi^{\mbX_i|\mbY=\mbY^k}\propto  \exp \left[-f_i(x_i)-\sum_{j=1}^{m}\frac{\sigma_{ij}}{2\eta}\Vert x_i-\mbY_j^k\Vert^2_2\right] $.
        \ENDFOR
    \ENDFOR
\end{algorithmic}
\end{algorithm}
One can see that Step 1 and Step 2 in the above Algorithm \ref{algo:algo2} can be executed in a distributed manner thanks to the bipartite graph structure and Gibbs samplers. 
Notice that in Step 1, 
\begin{equation} \label{eq:onestepingraph}
    \exp \left[-g_j(y_j)-\sum_{i=1}^{n}\frac{\sigma_{ij}}{2\eta}\Vert \mbX_i^k-y_j\Vert^2_2\right] \propto \exp \left[-g_j(y_j)-\frac{\sum_{i=1}^n\sigma_{ij}}{2\eta}\left\Vert y_j-\frac{\sum_{i=1}^n\sigma_{ij}\mbX_i^k}{\sum_{i=1}^n\sigma_{ij}}\right\Vert^2_2\right],
\end{equation}
so we can employ a similar framework in Section \ref{section:twonodes} to analyze this general case, which also holds for Step 2. More specifically, in Section \ref{section:twonodes}, we construct a diffusion process whose initial distribution is the marginal of $\pi^{\mbX\mbY}$, while in this section, the initial distribution is the pushforward measure of marginals.
\begin{theorem}\label{theorem:bipartite}
Denote the probability distribution of $(\mbX^k,\mbY^k)$ by $\mu^{\mbX\mbY}_k$. Then under Assumption \ref{assum:assum2}, for the target distribution \eqref{eq:target}, we have
\begin{equation}\nonumber
    D_{\rm KL}(\mu_k^{\mbX\mbY} \| \pi^{\mbX\mbY}) \leq \exp(-kC)D_{\rm KL}(\mu_0^{\mbX} \| \pi^{\mbX})
\end{equation}
with $C$ being defined as 
\begin{align}\label{eq:C}
    \begin{split}
        C & := \frac{\eta}{n}\int_0^1  \left[\frac{\eta t(1-t)}{\min_j\sum_{i=1}^n\sigma_{ij}}+\frac{mn(1-t)^2}{\min_i \alpha_{f_i}}+\frac{t^2}{\min_j \alpha_{g_j}}\right]^{-1} \rd t\\
        & \quad +  \frac{\eta}{m}\int_0^1  \left[\frac{\eta t(1-t)}{\min_i\sum_{j=1}^m\sigma_{ij}}+\frac{mn(1-t)^2}{\min_j \alpha_{g_j}}+\frac{t^2}{\min_i \alpha_{f_i}}\right]^{-1} \rd t.
    \end{split}
\end{align}
\end{theorem}
\begin{proof}
The proof of Theorem \ref{theorem:bipartite} is essentially a generalization of the proof of the two-node case in Section \ref{section:twonodes}. Inspired by \eqref{eq:onestepingraph}, we compare $D_{\mathrm{KL}}(\mu_k^{A\mbX} || \pi^{A\mbX})$ and $D_{\mathrm{KL}}(\mu_k^{\mbX} || \pi^{\mbX})$ where the $j$-th entry of $A\mbX = \nicefrac{(\sum_{i=1}^n\sigma_{ij}\mbX_i^k)}{(\sum_{i=1}^n\sigma_{ij})}$. According to  \eqref{eq:KLA} in Lemma \ref{lemma:KL}, we have 
\begin{equation}\label{eq:bipartite_1}
    D_{\mathrm{KL}}(\mu_k^{A\mbX} \| \pi^{A\mbX}) \leq D_{\mathrm{KL}}(\mu_k^{\mbX} \| \pi^{\mbX}).
\end{equation} 
Then by Lemma \ref{lemma:diffusion} and \eqref{eq:onestepingraph}, the transition occurring in Step 1 can be described by the following diffusion processes, 
\begin{equation}\nonumber
    \rd Z^j_t = \eta_j\nabla \log \phi^j_t(Z^j_t)\rd t + \sqrt{\eta_j} \rd W_t,\ Z^j_0 = \frac{\sum_{i=1}^n\sigma_{ij}\mbX_i^k}{\sum_{i=1}^n\sigma_{ij}}, \: \forall  j=1,2,\ldots,m.
\end{equation}
Here $\eta_j = \nicefrac{\eta}{\sum_{i=1}^n \sigma_{ij}}$.
Let $\mbZ_t=(Z_t^1,\ldots,Z_t^m)$. 
Denote the distribution of $\mbZ_t$ as $\pi_t$ if $\mbX^k \sim \pi^\mbX$. Clearly, $\pi_0 \sim \pi^{A\mbX}$ and $\pi_1 \sim \pi^{\mbY}$. It is noteworthy to mention that typically $A\mbX$ is not absolutely continuous w.r.t. Lebesgue measure, but $\pi_t$ still has a density representation.
By Lemma \ref{lemma:diffusion},  the density of $\pi_t$ satisfies
\begin{align}\nonumber
    \begin{split}
        \pi_t  &= \int \pi(\mbZ_t |\mbX,\mbY)\rd \pi^{\mbX \mbY}
        = \int \pi(\mbZ_t |A\mbX,\mbY) \rd \pi^{\mbX \mbY} \\
        &\propto \int \exp\left[-\sum_{i=1}^n f_i(x_i)-\sum_{j=1}^m g_j(y_j)-\sum_{i=1}^n\sum_{j=1}^m\frac{\sigma_{ij}}{2\eta} \Vert x_i-y_j \Vert^2_2\right] \\
        & \quad\exp\left[-\sum_{j=1}^{m}\frac{\sum_{i=1}^n\sigma_{ij}}{2t(1-t)\eta}\left\Vert z_t^j-ty_j-(1-t) \frac{\sum_{i=1}^n\sigma_{ij}x_i}{\sum_{i=1}^n\sigma_{ij}}\right\Vert_2^2\right]\rd \mbX \rd \mbY.
    \end{split}
\end{align}
By Proposition \ref{prop:prop1}, we have $\pi_t$ is strongly log-concave with a coefficient 
\begin{align}\label{eq:c_t}
    \begin{split}
            c_t &=  \left[\frac{\eta t(1-t)}{\min_j\sum_{i=1}^n\sigma_{ij}}+\frac{mn(1-t)^2}{\min_i \alpha_{f_i}}+\frac{t^2}{\min_j \alpha_{g_j}}\right]^{-1}.
    \end{split}
\end{align}
Observe that $\eta_j \geq \frac{\eta}{n}, \: \forall j$. Hence, by Theorem \ref{theorem:dataprocessing} and a lower bound of $c_t$ in \eqref{eq:c_t}  we obtain 
\begin{align}\nonumber
    \begin{split}
         D_{\mathrm{KL}} (\mu_k^{\mbY} || \pi^{\mbY}) &\le \exp\left(-\frac{\eta}{n}\int_0^1 c_t  \, \rd t\right) D_{\mathrm{KL}}(\mu_k^{A\mbX} \| \pi^{A\mbX})\\
         & \le \exp\left(-\frac{\eta}{n}\int_0^1 c_t \rd t \right)D_{\mathrm{KL}}(\mu_k^{\mbX} \| \pi^{\mbX})
    \end{split}
\end{align}
where the last inequality is from \eqref{eq:bipartite_1}.
By a symmetry analysis, one can get
\begin{align}\nonumber
    \begin{split}
         D_{\mathrm {KL}} (\mu_{k+1}^{\mbX} || \pi^{\mbX}) &\le \exp\left(-\frac{\eta}{m}\int_0^1  \left[\frac{\eta t(1-t)}{\min_i\sum_{j=1}^m\sigma_{ij}}+\frac{mn(1-t)^2}{\min_j \alpha_{g_j}}+\frac{t^2}{\min_i \alpha_{f_i}}\right]^{-1} \, \rd t\right)  D_{\mathrm{KL}}(\mu_k^{\mbY} \| \pi^{\mbY}).
    \end{split}
\end{align}
Lastly, with \eqref{eq:marginal} in Lemma \ref{lemma:KL},
\begin{equation}\nonumber
    D_{\mathrm{KL}}(\mu_k^{\mbX\mbY} \| \pi^{\mbX\mbY}) \leq \exp(-kC)D_{\mathrm{KL}}(\mu_0^{\mbX} \| \pi^{\mbX})
\end{equation}
where the linear convergence rate $C$ is defined in \ref{eq:C}.
\end{proof}
This convergence rate implies that the convergence depends on the worst strongly-convexity constant. For the impact of $\eta$, we have the same observation as in the two-node case,\textit{ i.e.}, larger $\eta$ implies faster convergence speed.   
    The convergence rate in Theorem \ref{theorem:bipartite} is not tight. For users who know the exact structure of the bipartite graph, the strong convexity constant $c_t$ can be further improved and thereby the convergence rate $\exp(-C)$. The tight log-Sobolev constant under the strong convexity assumption is the convexity constant of $\mbZ$-marginal distribution in Proposition \ref{prop:prop1}. A naive way to find the tight constant is to compute the Hessian matrix of the potential of joint density $\pi_t(\mathbf{x},\mathbf{y},\mathbf{z})$. The lack of tightness is partially due to the matrix inequalities in the proof of Proposition \ref{prop:prop1}.

\section{Convergence analysis with small $\eta$ for composite sampling}\label{section:small_eta}
Inspired by the success of the proximal sampling framework, we postulate that if $\eta$ is sufficient small, our target distribution \eqref{eq:target} would converge to 
\begin{equation}\nonumber
    \exp\left(-\sum_{i=1}^n f_i(x)-\sum_{j=1}^m g_j(x)\right).
\end{equation}
Due to the difficulty of the non-asymptotic behavior of \eqref{eq:target}, in this section, we focus on the two-node case. In this case, with the additional Gaussian component, \eqref{eq:target_twonodes} is a better-conditioned distribution than $\exp(-f(x)-g(x))$. A similar observation was initially  proposed by \citet{lee2021structured} and induces efficient sampling algorithms for structured log-concave distributions. The section is also a direct generalization of the one-node case in ~\citet{vono2022efficient}. For \eqref{eq:target_twonodes}, we will show the marginal of our target density \eqref{eq:target_twonodes} actually converges to the composite density with a sufficiently small $\eta$, i.e.,
\begin{equation}\nonumber
    \lim_{\eta \rightarrow 0} {\mathrm{TV}}(\pi_{\eta}^X,\nu)=0
\end{equation}  where $\pi_{\eta}^X$ is the $X$-marginal distribution of \eqref{eq:target_twonodes} and $\nu \propto \exp[-f(x)-g(x)]$.
Intuitively, if $\eta$ is small, then the density of $\pi_{\eta}$ will concentrate on the area, $x=y$. Therefore the marginal distribution of $\pi_{\eta}$ would approximately be $\nu$. To begin with, we prove the asymptotic property in the following Proposition \ref{prop:eta2}.
\begin{proposition}\label{prop:eta2}
Assume $\exp(-g) \in L_1(\bbR^d)$ and $\exp(-f)$ is essentially bounded from above, then for
\begin{equation}\nonumber
\pi^{X}_{\eta} \propto \int\exp\left[-f(x)-g(y)-\frac{1}{2\eta}\|x-y\|^2_2\right] \rd y
\end{equation} and
\begin{equation}\nonumber
\nu \propto \exp\left[-f(x)-g(x)\right],
\end{equation}
we have 
    \begin{equation}\nonumber
        \lim_{\eta\rightarrow 0} \mathrm{TV}(\pi_\eta^X,\nu) =0.
    \end{equation}
\end{proposition}
\begin{proof}
 Under the condition that $\exp(-g) \in L_1(\bbR^d)$, one can show the convolution $\exp(-g) * \mN(0,\eta\ide)$ converges to $\exp(-g)$ almost everywhere (Theorem 1 in \citet{bear1979approximate}), which yields
\begin{equation}\label{eq:eq_asympoticial_1}
    \lim_{\eta \rightarrow 0} \frac{1}{\sqrt{2\pi\eta}}\int \exp\left[-f(x)-g(y)-\frac{1}{2\eta}\|x-y\|_2^2\right]\rd y  = \exp[-f(x)-g(x)] \quad a.e.
\end{equation}
Meanwhile, as  $\exp(-g) \in L_1(\bbR^d)$, by a standard argument (Theorem 9.1.11 in \citet{Hei19}), one has $\|\exp(-g) * \mN(0,\eta\ide)-\exp(-g)\|_1 \rightarrow 0$. As $\exp(-f)$ is essentially bounded from above, it follows that
\begin{equation}\label{eq:eq_asympoticial_2}
     \lim_{\eta \rightarrow 0} \frac{1}{\sqrt{2\pi\eta}}\int \exp\left[-f(x)-g(y)-\frac{1}{2\eta}\|x-y\|_2^2\right]\rd y\rd x  = \int\exp[-f(x)-g(x)]\rd x.
\end{equation}
With \eqref{eq:eq_asympoticial_1}, \eqref{eq:eq_asympoticial_2} and  Scheff\`{e}'s lemma \citep{scheffe1947useful}, we have
\begin{equation}\nonumber
    \lim_{\eta \rightarrow 0}\mathrm{TV}(\pi_\eta^X,\nu) = 0.
\end{equation}
\end{proof}
However, the non-asymptotic behaviour of $\pi_{\eta}^{X} \rightarrow \nu$ is challenging. In this section, we summarize and analyze four special cases to promote the study of it. The idea of bounding the distance of $\pi_{\eta}^X$ and $\nu$ is initially proposed in \citet{lee2021structured}, which we summarize in Example \ref{ex:ex1}.
\begin{example}\label{ex:ex1}
Consider a joint distribution,
\begin{equation}\nonumber
    \pi_{\eta}^{XY} \propto \exp[-f(x)-g(y)-\nicefrac{1}{2\eta}\Vert x-y\Vert^2_2].
\end{equation}
Assume $g(y)=0$. Then $\pi^{X}_{\eta} \propto \exp[-f(x)]$. In this case, the $X$-marginal distribution is exactly $\nu$.
\end{example}
To proceed, we analyze a symmetric case where $f(x) = 0$. This example resembles the setting of ~\citet{vono2022efficient}, and we establish a better dimension dependency than~\citet{vono2022efficient}: from $d$ to $d^{1/2}$. 
\begin{example}\label{ex:ex2}
Consider a joint distribution,
\begin{equation}\nonumber
    \pi_{\eta}^{XY} \propto \exp[-f(x)-g(y)-\nicefrac{1}{2\eta}\Vert x-y\Vert^2_2]
\end{equation} where $x, y \in \bbR^d$.
Assume $f(x)=0$ and $g$ is $\alpha$-strongly convex and $\beta$-smooth. Then
\begin{equation}\nonumber
    D_{\mathrm{KL}}(\pi_\eta^X|| \nu) \leq \frac{d\alpha}{2\beta}\left[\exp\left(\frac{2\beta^2\eta}{\alpha}\right)-1\right].
\end{equation}
Moreover, with Pinsker's inequality, we obtain
\begin{equation}\nonumber
    {\mathrm{TV}}(\pi_\eta^X, \nu) \leq \sqrt{\frac{d\alpha}{4\beta}\left[\exp\left(\frac{2\beta^2\eta}{\alpha}\right)-1\right]}.
\end{equation}
\end{example}
\begin{proof}
 Appendix \ref{subappendix:ex2}.
\end{proof}
Even though  we already find the convergence rate for the particular cases with only one node, \textit{i.e.}, Example \ref{ex:ex1} and \ref{ex:ex2}, the analysis for the multi-variable case is demanding.  It is worth noting that instead of using the potential $f(x)+g(y)+\frac{1}{2\eta}\Vert x-y \Vert_2^2$, \citet{lee2021structured} constructs a new potential with an additional quadratic term,  $f(x)+g(y)+\frac{1}{2\eta}\Vert x-y \Vert_2^2+\frac{\eta \alpha_f^2}{2}\Vert x-x^* \Vert_2^2$, where $x^*$ is the minimizer of $f+g$. Our analysis is different from this previous work in two aspects. First, our target density does not include the $L_2$ distance of $x$ and $x^*$. Secondly, we find that without the additional term,  it is impossible to uniformly bound the total variation  distance for strongly convex and smooth potentials by a single $\eta$. Thus, for the non-asymptotic bound, we illustrate a necessary condition by Example \ref{ex:ex3}.
\begin{example}\label{ex:ex3}
Consider a joint distribution,
\begin{equation}\nonumber
    \pi_{\eta}^{XY} \propto \exp[-f(x)-g(y)-\nicefrac{1}{2\eta}( x-y)^2].
\end{equation}
Let $f(x) = (x-u_1)^2$ and $g(x) = (x - u_2)^2$. Then 
 \begin{align}\nonumber
    \begin{split}
     \pi_\eta^X &= \mN\left(\frac{u_2+(2\eta+1)u_1}{2\eta+2},\frac{2\eta+1}{4\eta+4}\right), \\
     \nu &= \mN\left(\frac{u_1+u_2}{2},\frac{1}{4}\right).
    \end{split}
\end{align}
By Theorem 1.3 in \citet{devroye2018total}, we obtain  
\begin{equation}\nonumber
    {\mathrm{TV}}(\pi_\eta^X, \nu) \ge \min \left\{\frac{\eta|u_1-u_2|}{5\eta+5}, 1/200\right\}.
\end{equation}
This means the convergence rate relies on $|u_1-u_2|$.
Note that even with other metrics, the term $|u_1-u_2|$ would affect the non-asymptotic behavior. For instance, 
\begin{equation}\nonumber
    D_{\mathrm{KL}}(\nu || \pi_\eta^X) = \frac{1}{2}\log\frac{2\eta+1}{\eta+1} -\frac{\eta}{4\eta+2}+\frac{\eta^2(u_1-u_2)^2}{(2\eta+1)(2\eta+2)},
\end{equation}
and 
\begin{equation}\nonumber
    W_2^2(\pi_\eta^X,\nu) = \left(\sqrt{\frac{2\eta+1}{4\eta+4}} - \frac{1}{2}\right)^2 + \frac{\eta^2(u_1-u_2)^2}{(2\eta+2)^2}.
\end{equation}
\end{example}
Hence, to bound the distance of $\pi_\eta^X$ and $\nu$, we make Assumption \ref{assu:assu3}.
\begin{assumption}\label{assu:assu3}
Assume 
    $f$ and $g$ share the same minimizer, i.e., $\nabla f(x^*) = \nabla g(x^*)=0$.
\end{assumption}
\begin{remark}
    The condition that $f$ and $g$ share the same minimizer can be achieved by shifting $f$ and $g$ by linear terms \citep{lee2021structured}, i.e., $\bar{f}(x) = f(x)-\langle \nabla f(x^*),x \rangle$ and $\bar{g}(x) = g(x)+\langle \nabla f(x^*),x \rangle$ where $x^*$ is the stationary point of $f+g$. Notice that $\exp[-f(x)-g(x)] = \exp[-\bar{f}(x)-\bar{g}(x)]$ and $\nabla \bar{f}(x^*) = \nabla \bar{g}(x^*)$.
\end{remark}
\begin{proposition}\label{proposition:quadratic}
    Consider a joint distribution,
    \begin{equation}\nonumber
        \pi_{\eta}^{XY} \propto \exp[-f(x)-g(y)-\nicefrac{1}{2\eta}\Vert x-y\Vert^2_2].
    \end{equation}
    Assume $g(y) = \frac{1}{2\sigma^2}\|y-u\|_2^2$
    and $f(x)$ is $\alpha_f$-strongly convex.
    Then with Assumption \ref{assu:assu3}, we have
    \begin{equation}\nonumber
         D_{\mathrm{KL}}(\nu|| \pi_{\eta}^X) \le \frac{d\eta^2}{2(\alpha_f\sigma^4+\sigma^2)^2}.
    \end{equation}
    Moreover, by Pinsker's inequality, we have
    \begin{equation}\nonumber
        {\mathrm{TV}}(\pi_\eta^X, \nu) \leq \frac{\eta\sqrt{d}}{2(\alpha_f\sigma^4+\sigma^2)}.
    \end{equation}
\end{proposition}
\begin{proof}
We build a continuous path to connect $\pi_\eta^X$ and $\nu$ and then establish an inequality for the first-order and second-order time derivatives of $D_{\mathrm{KL}}(\pi_0 || \pi_t)$.
See Appendix \ref{subappendix:ex4} for the full proof.
\end{proof}

\section{Discussion}
\label{sec:discussion}
In this work, we establish the first non-asymptotic bound of a Gibbs sampler for structured log-concave distributions over bipartite graphs.  For the two-node graph, we further prove its asymptotic property to a lower-dimensional composite distribution and show the convergence rate for special cases. Here are some directions for future investigation. 
\begin{itemize}
    \item How to compute the non-asymptotic rate for \eqref{eq:target} for non-convex potentials? It is possible to generalize to the LSI assumption with the same techniques in this work. However, this requires a tight LSI constant as discussed in Remark \ref{remark:LSIAssump}.
    \item How to analyze the non-asymptotic bound of Gibbs sampling on general graphs? The analysis for the bipartite graphs is a natural generalization of the two-node case due to \eqref{eq:onestepingraph}. Whether the same techniques can be adopted for general graphs is not clear.
    \item How to find the convergence rate of general distributions in the form of \eqref{eq:target_twonodes} as $\eta$ approaches 0? Solving this question will lead to the non-asymptotical bounds for much more general distributions.
\end{itemize} 

\section{Acknowledgments}
We extend our gratitude to the anonymous reviewers for their
invaluable feedback that enhanced this manuscript. Financial support from NSF under grants 1942523, 2008513, and 2206576 is greatly acknowledged.

\newpage
\bibliography{reference}

\appendix
\newpage

\section{Supplementary lemmas}
\label{section:appendix}

\subsection{Proof of Lemma \ref{lemma:basic}}
\label{subsection:appendix_A1}
\begin{proof}
We start with the time derivative of $D_{\mathrm{KL}}(\mu_t || \pi_t)$. With  the Fokker–Planck equation, we have
    \begin{align}\label{eq:dataprocessing1}
    \begin{split}
        \frac{\partial D_{\mathrm{KL}}(\mu_t || \pi_t) }{\partial t} 
        &=\int \partial_{t}\mu_t\log\frac{\mu_t}{\pi_t} + \partial_{t}\left(\frac{\mu_t}{\pi_t}\right)\pi_t \\
        &=\int \partial_{t}\mu_t\log\frac{\mu_t}{\pi_t} + \int \partial_{t} \mu_t - \int \partial_{t}\pi_t\frac{\mu_t}{\pi_t} \\
        &= \int \left[-\nabla \cdot ( b_t\mu_t) + \sum_i\frac{\sigma_{t,i}^2}{2}\partial^2_i \mu_t\right]\log\frac{\mu_t}{\pi_t}\\
        &- \int \left[-\nabla \cdot ( b_t\pi_t) + \sum_i\frac{\sigma_{t,i}^2}{2}\partial^2_i \pi_t\right]\frac{\mu_t}{\pi_t},
    \end{split}
    \end{align}
where the symbol $\partial^2_i$ is the second-order derivative w.r.t. the $i$-th coordinate.
By integration by parts, we obtain
    \begin{align}\label{eq:dataprocessing2}
    \begin{split}
        \int \nabla \cdot ( b_t\mu_t)\log\frac{\mu_t}{\pi_t}  &= -\int  \mu_tb_t\cdot \nabla \log\frac{\mu_t}{\pi_t}\\
        &= -\int  \pi_tb_t\cdot \nabla \frac{\mu_t}{\pi_t}\\
        &= \int  \nabla \cdot( b_t\pi_t) \frac{\mu_t}{\pi_t}.
    \end{split}
    \end{align}
Combining \eqref{eq:dataprocessing1} and \eqref{eq:dataprocessing2} yields 
    \begin{align}\label{eq:dataprocessing3}
    \begin{split}
        \frac{\partial D_{\mathrm{KL}}(\mu_t || \pi_t) }{\partial t} 
        &= \sum_i\frac{\sigma_{t,i}^2}{2}\int \partial^2_i \mu_t\log\frac{\mu_t}{\pi_t} -\partial^2_i \pi_t \frac{\mu_t}{\pi_t}\\
        &= -\sum_i\frac{\sigma_{t,i}^2}{2}\int \partial_i \mu_t \left(\partial_i\log\frac{\mu_t}{\pi_t}\right) -\partial_i \pi_t \left(\partial_i\frac{\mu_t}{\pi_t}\right)\\
        &= -\sum_i\frac{\sigma_{t,i}^2}{2}\int \left\Vert\partial_i\log\frac{\mu_t}{\pi_t}\right\Vert_2^2 \mu_t\\
        &\le-\frac{\lambda_t^2}{2}I(\mu_t||\pi_t)
    \end{split}
    \end{align}
where $\lambda_t^2 \leq \min_i\sigma_{t,i}^2$.
\end{proof}

\subsection{Proof of Lemma \ref{lemma:diffusion}}
\label{subsection:appendix_A2}
\begin{proof}
We show $P(z(t)) \propto \phi_t (z(t)) \hat \phi_t (z(t))$ first. By definition of Brownian bridge, the conditional distribution of $z(t)$ on $z(0), z(1)$ is
    \[
        P(z(t) | z(0), z(1)) \propto \exp(-\frac{1}{2\eta t(1-t)}\|z(t)-(1-t)z(0)-tz(1)\|^2).
    \]
It follows that
    \begin{eqnarray*}
        P(z(t)) &\propto& \int \exp(-f(x)-g(y)-\frac{1}{2\eta}\|x-y\|^2-\frac{1}{2\eta t(1-t)}\|z(t)-(1-t)x-ty\|^2) \rd x \rd y\\
        &\propto& \phi_t (z(t)) \hat \phi_t (z(t)).
    \end{eqnarray*}
Now consider the forward process \eqref{eq:forward}. It can be shown \citep{tzen2019theoretical,zhangpath2022} that the conditional distribution of $Z_1$ on $Z_0$ is 
    \[
        P(Z_1 | Z_0) \propto \exp(-g(Z_1) - \frac{1}{2\eta}\|Z_1-Z_0\|^2),
    \]
and \eqref{eq:forward} belongs to the same reciprocal class \citep{jamison1974reciprocal} as the Brownian bridge, that is, $P(Z_\cdot|Z_0, Z_1)$ coincides with that of the Brownian bridge $P(z(\cdot)| z(0), z(1))$. The drift $\phi_t$ is known to be the Follmer drift \citep{jamison1974reciprocal}. Taking into account the initial condition $Z_0\sim P(z(0))$, we conclude that the diffusion process \eqref{eq:forward} is a representation of the trajectory distribution we want. 

By Doob's $h$-transform \citep{anderson1982reverse,chen2016relation}, \eqref{eq:forward} is has an equivalent backward representation
    \[
        \rd Z_t = [\eta \nabla \log \phi_t (Z_t) - \eta \nabla \log P(Z_t)] \rd t +\sqrt{\eta} \rd W_t.
    \]
Invoking the fact that $P(z(t)) \propto \phi_t (z(t)) \hat \phi_t (z(t))$, we arrive at the backward representation \eqref{eq:backward}.
\end{proof}

\subsection{KL divergence contraction }

\begin{lemma}\label{lemma:KL}
Let $\pi^{XY}$ and $\mu^{XY}$ be any two joint distributions of $(x,y)$. Then we have
\begin{equation}\label{eq:marginal}
    D_{\mathrm{KL}}(\mu^{XY}||\pi^{XY}) = D_{\mathrm{KL}}(\mu^X||\pi^{X}) + \bbE_{\mu^X}(D_{\mathrm{KL}}(\mu^{Y|X}||\pi^{Y|X}))
\end{equation}
Note that if we further assume $\mu^{Y|X}=\pi^{Y|X}$, then 
\begin{equation}
    D_{\mathrm{KL}}(\mu^{XY}||\pi^{XY}) = D_{\mathrm{KL}}(\mu^X||\pi^{X})
\end{equation}
A more general form is $A \in \bbR^{m \times n}$ and $X \in \bbR^n$. For any two distributions $\mu^{X}$ and $\pi^{X}$, and their corresponding pushforward measures $\mu^{AX}$ and $\pi^{AX}$, we have 
\begin{equation}\label{eq:KLA}
    D_{\mathrm{KL}}(\mu^{AX}||\pi^{AX}) \leq D_{\mathrm{KL}}(\mu^X||\pi^{X})
\end{equation}
\end{lemma}

\begin{proof}
\begin{align*}
D_{\mathrm{KL}}(\mu^{XY}||\pi^{XY})&= \int \log \frac{\mu^{XY}}{\pi^{XY}}\mu^{XY}\\
&= \int \log \left(\frac{\mu^X}{\pi^{X}}\frac{\mu^{Y|X}}{\pi^{Y|X}}\right)\mu^{X}\mu^{Y|X}  \\
&=\int \log \frac{\mu^X}{\pi^{X}}\mu^{X}\mu^{Y|X} +  \int \log \frac{\mu^{Y|X}}{\pi^{Y|X}}\mu^{X}\mu^{Y|X}\\
&=\int \log \frac{\mu^X}{\pi^{X}}\mu^{X} + \int D_{\mathrm{KL}}(\mu^{Y|X}||\pi^{Y|X})\mu^{X} \\
&=D_{\mathrm{KL}}(\mu^X||\pi^{X}) + \mathbb{E}_{\mu^X}(D_{\mathrm{KL}}(\mu^{Y|X}||\pi^{Y|X})).
\end{align*}
We note that the decomposition above still holds if we treat $\frac{\mu^{XY}}{\pi^{XY}}$ as the  Radon-Nikodym derivative instead of the ratio of density functions. See Lemma 1 in \citet{vargas2021machine} for details.

Next we prove \eqref{eq:KLA}. For the more general case, initially, we assume $A$ is full row rank, which implies $m\leq n$. Then let $\hat{A}$ be an $n \times n$ invertible matrix whose first $m$ rows exactly comprise $A$. It follows that 
\begin{align*}
    D_{\mathrm{KL}}(\mu^{\hat{A}X}||\pi^{\hat{A}X}) 
    &= \int \log \frac{\mu^{\hat{A}X}}{\pi^{\hat{A}X}}\mu^{\hat{A}X}\\
    &= \int \log \frac{\mu^{X}(\hat{A}^{-1}x)}{\pi^{X}(\hat{A}^{-1}x)}|\det(\hat{A}^{-1})|\mu^{X}(\hat{A}^{-1}x)\\
    &= \int \log \frac{\mu^{X}(y)}{\pi^{X}(y)}\mu^{X}(y)\\
    &=  D_{\mathrm{KL}}(\mu^{X}||\pi^{X}).
\end{align*}
According to Equation \eqref{eq:marginal}, since $AX$ is just the first $m$ components of $\hat{A}X$, we have 
\begin{equation*}
    D_{\mathrm{KL}}(\mu^{AX}||\pi^{AX}) 
    \leq D_{\mathrm{KL}}(\mu^{\hat{A}X}||\pi^{\hat{A}X})  
    =  D_{\mathrm{KL}}(\mu^{X}||\pi^{X}).
\end{equation*}

Next consider the case where $A$ is not full row rank. One can show, even in this case, $AX$ is still a well-defined random vector, but $\mu^{AX}$ and $\pi^{AX}$ are not absolutely continuous with respect to the Lebesgue measure. Define the rank of $A$ by $r$. One can always pick up $r$ independent rows from $A$ which we denote by $\bar{A}$. Notice that the probability measures over $\bar{A}X$ is essentially the same as the measure over $AX$, as all other coordinates are uniquely determined by $\bar{A}X$. Hence, we have
\begin{equation*}
    D_{\mathrm{KL}}(\mu^{AX}||\pi^{AX}) = D_{\mathrm{KL}}(\mu^{\bar{A}X}||\pi^{\bar{A}X}).
\end{equation*}
Because $\bar{A}$ is full row rank, according to the previous case, we eventually have
\begin{equation*}
    D_{\mathrm{KL}}(\mu^{AX}||\pi^{AX}) \leq D_{\mathrm{KL}}(\mu^{X}||\pi^{X}).
\end{equation*}

\end{proof}

\section{Strong-concavity constant}

\begin{proposition}\label{prop:prop1}
With Assumption \ref{assum:assum2}, define
\begin{align}
    \begin{split}
        \pi_t(\mathbf{x},\mathbf{y},\mathbf{z})  &\propto \exp\left[-\sum_{i=1}^n f_i(x_i)-\sum_{j=1}^m g_j(y_j)-\sum_{i=1}^n\sum_{j=1}^m\frac{\sigma_{ij}}{2\eta} \Vert x_i-y_j \Vert^2_2\right] \\
        & \quad\exp\left[-\sum_{j=1}^{m}\frac{\sum_{i=1}^n\sigma_{ij}}{2t(1-t)\eta}\|z_t^j-ty_j-(1-t) \frac{\sum_{i=1}^n\sigma_{ij}x_i}{\sum_{i=1}^n\sigma_{ij}}\|_2^2\right].
    \end{split}
\end{align}
Then its $\mathbf{z}$-marginal distribution, $\pi_t(\mathbf{z}) \propto \int  \pi_t(\mathbf{x},\mathbf{y},\mathbf{z})\rd\mathbf{x} \rd\mathbf{y}$, is strongly log-concave with a coefficient $ c_t = \left[\frac{\eta t(1-t)}{\min_j\sum_{i=1}^n\sigma_{ij}}+\frac{mn(1-t)^2}{\min_i \alpha_{f_i}}+\frac{t^2}{\min_j \alpha_{g_j}}\right]^{-1}$.
\end{proposition}
\begin{proof}
    We will begin with proving the joint distribution $\pi_t(\mathbf{x},\mathbf{y},\mathbf{z})$ is strongly concave. As each $f_i$ and each $g_j$ are strongly convex with coefficients $\alpha_{f_j}$ and $\alpha_{g_j}$, respectively, we can rewrite $f_i(x_i) = \hat{f}_i(x_i)+\frac{\alpha_{f_i}}{2}\|x_i\|^2_2$ and $g_j(y_j) = \hat{g}_j(y_j)+\frac{\alpha_{g_j}}{2}\|y_j\|^2_2$ where $\hat{f}_i$ and $\hat{g}_j$ are convex functions.
    Denote $(x_1,\ldots,x_n)$, $(y_1,\ldots,y_m)$, $(z_1,\ldots,z_m)$ by  $\mathbf{x},\mathbf{y},\mathbf{z}$, respectively. Then we can decompose $\pi_t(\mathbf{x},\mathbf{y},\mathbf{z}) $ as $-\log\pi_t(\mathbf{x},\mathbf{y},\mathbf{z})
             =F(\mathbf{x},\mathbf{y},\mathbf{z})+G(\mathbf{x},\mathbf{y},\mathbf{z})$
    where\begin{equation}
        F(\mathbf{x},\mathbf{y},\mathbf{z})=\sum_{i=1}^n \hat{f}_i(x_i) +\sum_{j=1}^m \hat{g}_j(y_j) +\sum_{i=1}^n\sum_{j=1}^m\frac{\sigma_{ij}}{2\eta} \Vert x_i-y_j \Vert^2_2
    \end{equation} 
    and
    \begin{equation}
        G(\mathbf{x},\mathbf{y},\mathbf{z})=\sum_i \frac{\alpha_{f_i}}{2}\|x_i\|^2_2 + \sum_j \frac{\alpha_{g_j}}{2}\|y_j\|^2_2 
 +\sum_{j=1}^{m}\frac{\sum_{i=1}^n\sigma_{ij}}{2t(1-t)\eta}\left\Vert z_t^j-ty_j-(1-t) \frac{\sum_{i=1}^n\sigma_{ij}x_i}{\sum_{i=1}^n\sigma_{ij}}\right\Vert_2^2.
    \end{equation} 
    By definition, $F(\mathbf{x},\mathbf{y},\mathbf{z})$ is convex jointly in $\mathbf{x},\mathbf{y},\mathbf{z}$ while $G(\mathbf{x},\mathbf{y},\mathbf{z})$ is strongly-convex jointly in $\mathbf{x},\mathbf{y},\mathbf{z}$. Hence $\pi_t(\mathbf{x},\mathbf{y},\mathbf{z})$ is strongly log-concave. By theorem 3.8 in \citep{saumard2014log}, the $\mathbf{z}$-marginal distribution is also strongly log-concave. Then we will compute the strongly-convexity constant for $\pi_t$.
    Define  
    \begin{equation}\nonumber
        M = 
        \left(
            \begin{array}{c|c}
            \ide_{m \times m} & B\\ \hline
            \mathbf{0}_{(m+n) \times m} & \ide_{(m+n) \times (m+n)}
            \end{array}
        \right)
    \end{equation}\nonumber
    where
    \begin{equation}\nonumber
      B = 
        \left(
            \begin{array}{c|c}
            -(1-t)A^T & -t\ide_{m \times m}
            \end{array}
        \right).
    \end{equation}
    and
    the entry of $A$ at the $i$-th row and $j$-th column is $\nicefrac{\sigma_{ij}}{\sum_{i=1}^n\sigma_{ij}}$.
    Furthermore, define a diagonal matrix $\Lambda$ as 
    \begin{equation}\nonumber
        \Lambda = 
            \begin{pmatrix}
                \Lambda_1& &\\
                &\Lambda_2 & \\
                & & \Lambda_3
            \end{pmatrix}
    \end{equation}
    with 
    \begin{align}
        \begin{split}
            \Lambda_1 &=\frac{1}{t(1-t)\eta}\mathrm{diag}\left(\sum_{i=1}^n\sigma_{i1},\ldots,\sum_{i=1}^n\sigma_{im}\right), \\
            \Lambda_2 &=\mathrm{diag}(\alpha_{f_1},\ldots,\alpha_{f_n}), \\
            \Lambda_3 &=\mathrm{diag}(\alpha_{g_1},\ldots,\alpha_{g_m}).
        \end{split}
    \end{align}
    Then clearly,
    \begin{equation}\nonumber
         G(\mathbf{x},\mathbf{y},\mathbf{z}) = \frac{1}{2}(\mathbf{z},\mathbf{x},\mathbf{y})^T M^T\Lambda M (\mathbf{z},\mathbf{x},\mathbf{y})^T.
    \end{equation} 
    By Theorem 3.8 in \citep{saumard2014log}, it is sufficient to compute the smallest eigenvalue of the upper left block of the inverse of $M^T\Lambda$.
    The inverse of the blocked matrix can be given by Woodbury matrix identity. Combining the two results, we only need to find $c_t >0$ such that  
    \begin{equation}\nonumber
        c_t \ide \preceq  (\Lambda_1^{-1}+B\Lambda_2^{-1}B^T)^{-1}=(\Lambda_1^{-1}+(1-t)^2A^T\Lambda_{2}^{-1}A+t^2\Lambda_{3}^{-1})^{-1}.
    \end{equation}
 Notice that the largest eigenvalue of $A^T\Lambda_{2}^{-1}A$ is the squared operator norm of $\sqrt{\Lambda_{2}^{-1}}A$, which satisfies
 \begin{equation}\nonumber
     \left\Vert \sqrt{\Lambda_{2}^{-1}}A\right\Vert_{\mathrm{op}}^2\le mn\left\Vert \sqrt{\Lambda_{2}^{-1}}A\right\Vert_{\mathrm{max}}^2 \le \frac{mn}{\min_i \alpha_{f_i}}.
 \end{equation}
 Therefore, 
  \begin{align}
    \begin{split}
         c_t = \left[\frac{\eta t(1-t)}{\min_j\sum_{i=1}^n\sigma_{ij}}+\frac{mn(1-t)^2}{\min_i \alpha_{f_i}}+\frac{t^2}{\min_j \alpha_{g_j}}\right]^{-1}.
    \end{split}
    \end{align}

\end{proof}

\section{Convergence rate for the two-node case under the convexity assumption}\label{appendix:convexity}
To begin with, we introduce the 2-Wasserstein distance. The 2-Wasserstein distance for two distributions $\mu$ and $\nu$ with finite second-order moments is defined as follows,
\begin{equation}\nonumber
    W_2^2(\mu,\nu) =\inf_{\gamma \in \Gamma(\mu,\nu) } \int_{(x,y)\sim \gamma} \|x-y\|_2^2 \rd \gamma(x,y)
\end{equation} where $\Gamma(\mu,\nu)$ is the set of all couplings for $\mu$ and $\nu$.
\begin{proposition}[Data processing inequality with log-concavity]\label{thm:dataprocessing2} 
    Under the condition of Lemma \ref{lemma:basic}, if we further assume $\pi_t$ is log-concave and 
    $W_2^2(\mu_t,\pi_t) \le C, \: \forall t\in [0,T]$, then 
    \begin{equation}\nonumber
        \frac{1}{D_{\mathrm{KL}}(\mu_T || \pi_T)}  \ge \frac{1}{D_{\mathrm{KL}}(\mu_0 || \pi_0)}+\frac{\int_0^T \lambda_t^2dt}{2C}.
    \end{equation}
\end{proposition}
\begin{proof}
Because of the log-concavity of $\pi_t$, one can show the convexity of $D_{\mathrm{KL}}(\cdot || \pi_t)$ along Wasserstein geodesics (see Theorem 9.4.11 in \citet{ambrosio2005gradient}), which follows that
\begin{equation}\nonumber
     0= D_{\mathrm{KL}}(\pi_t || \pi_t) \ge D_{\mathrm{KL}}(\mu_t || \pi_t) + \bbE_{x_t,y_t \sim \pi^*(\mu_t,\pi_t)}\langle \nabla \log \frac{\mu_t}{\pi_t}(x_t), x_t-y_t  \rangle .
\end{equation} where $\pi^*(\mu_t,\pi_t)$ is the optimal coupling in the sense of 2-Wasserstein distance.
Hence,
\begin{equation}\nonumber
     D_{\mathrm{KL}}(\mu_t || \pi_t)^2 \le \bbE_{\mu_t}\|\nabla \log \frac{\mu_t}{\pi_t}\|_2^2 W_2^2(\mu_t,\pi_t) \\
     = I(\mu_t || \pi_t)W_2^2(\mu_t,\pi_t).
\end{equation}
Combining with \eqref{eq:dataprocessing3} yields
    \begin{equation}\nonumber
        \frac{\partial D_{\mathrm{KL}}(\mu_t || \pi_t) }{\partial t}   \leq -\frac{\lambda_t^2}{2} \frac{D_{\mathrm{KL}}(\mu_t || \pi_t)^2}{W_2^2(\mu_t,\pi_t)} \leq -\frac{\lambda_t^2D_{\mathrm{KL}}(\mu_t || \pi_t)^2}{2C},
    \end{equation}
    which gives 
    \begin{equation}\nonumber
      \frac{1}{D_{\mathrm{KL}}(\mu_T || \pi_T)}  \ge \frac{1}{D_{\mathrm{KL}}(\mu_0 || \pi_0)}+\frac{\int_0^T \lambda_t^2dt}{2C}.
    \end{equation}
\end{proof}
We are ready to show the proof of Theorem \ref{theorem:convexity}.
\begin{proof}
We construct the diffusion process for both $(\pi^X,\pi^Y)$ and $(\pi_k^X,\pi_k^Y)$ as in the proof of \ref{theorem:twonode},
\begin{equation}\label{eq:Pro19_1}
        \rd Z_t = -\eta \nabla \log \phi_t (Z_t) \rd t + \sqrt{\eta} \rd W_t.
\end{equation}
Then we have the explicit expression of $\pi_t$, the distribution of $Z_t$ given $\pi_0 =\pi^X$, satisfies,
\begin{equation}\nonumber
   \pi_t \propto \int \exp[-g(y)-\frac{1}{2\eta(1-t)}\Vert z-y \Vert^2_2] \rd y \int \exp[-f(x)-\frac{1}{2\eta t}\Vert z-x \Vert^2_2] \rd x. 
\end{equation}
Notice that the drift term $ -\eta\nabla \log \phi_t$ in \eqref{eq:Pro19_1} is convex. Thus, with $W_2^2$-contraction , we have 
\begin{equation}\nonumber
    W_2^2(\mu_t,\pi_t) \le W_2^2(\mu_0,\pi_0) = W_2^2(\mu_0^{X},\pi^{X}).
\end{equation}
As $\pi_t$ is convex, which is due to the convexity assumption for $f$ and $g$ (see Theorem 3.3 in \citet{saumard2014log}),  we can apply Theorem  \ref{thm:dataprocessing2} with $C \equiv W_2^2(\mu_0^{X},\pi^{X})$ on \eqref{eq:Pro19_1}, which yields
\begin{equation}\label{eq:Prop_2}
    \frac{1}{D_{\mathrm{KL}}(\mu_k^Y ||\pi^Y)}  \ge \frac{1}{D_{\mathrm{KL}}(\mu_k^{X} || \pi^{X})}+\frac{\eta}{2 W_2^2(\mu_0^X,\pi^X)}.
\end{equation}
We can analyze Step 2 in Algorithm \ref{algo:algo1} with the same framework, which gives
\begin{equation}\label{eq:Prop_3}
    \frac{1}{D_{\mathrm{KL}}(\mu_{k+1}^X ||\pi^X)}  \ge \frac{1}{D_{\mathrm{KL}}(\mu_{k}^Y || \pi^Y)}+\frac{\eta}{2 W_2^2(\mu_0^X,\pi^X)}.
\end{equation}
With \eqref{eq:Prop_2}, \eqref{eq:Prop_3} and Lemma \ref{lemma:KL}, one has for any $k$
    \begin{equation}\nonumber
        \frac{1}{D_{\mathrm{KL}}(\mu^{XY}_k || \pi^{XY})}  \ge \frac{1}{D_{\mathrm{KL}}(\mu_0^X || \pi^{X})}+\frac{k\eta}{W_2^2(\mu_0^X,\pi^X)}.
    \end{equation}
It follows that 
    \begin{equation}\nonumber
        D_{\mathrm{KL}}(\mu^{XY}_k || \pi^{XY}) \le \frac{W_2^2(\mu_0^X,\pi^X)}{k\eta}.
    \end{equation}
\end{proof}

\section{Non-asymptotic convergence rate to composite distributions}\label{appendix:small_eta}

\subsection{Proof of Example \ref{ex:ex2}}\label{subappendix:ex2}

\begin{proof}
Note that with the condition $f=0$, $\pi^X_\eta$ = $\nu * \mN(0,\eta\ide)$. Thus, $\forall t \in (0,\eta]$, we define $\pi_t = \pi_0 * \mN(0,t\ide)$. Then it is sufficient to bound the distance between $\pi_t$ and $\pi_0$. By definition, $\frac{\partial \pi_t}{\partial t} = \frac{1}{2}\Delta \pi_t$. It follows that 
\begin{align}\nonumber
    \begin{split}
    \frac{\partial D_{\mathrm{KL}}(\pi_t|| \pi_0)}{\partial t} 
    & =  \int \frac{\partial \pi_t}{\partial t} + \log \frac{\pi_t}{\pi_0}\frac{\partial \pi_t}{\partial t} \\
    & = -\frac{1}{2}\int \pi_t \langle \nabla \log \pi_t, \nabla \log \frac{\pi_t}{\pi_0}\rangle \\
    & = -\frac{1}{2}\int \pi_t \|\nabla \log \frac{\pi_t}{\pi_0}\|^2_2  - \frac{1}{2}\int \pi_t \langle \nabla \log \pi_0, \nabla \log \frac{\pi_t}{\pi_0} \rangle \\
    & = -\frac{1}{2}\int \pi_t \|\nabla \log \frac{\pi_t}{\pi_0}\|^2_2 + \frac{1}{2}\int \pi_t  \|\nabla \log \pi_0\|^2 - \frac{1}{2}\int \langle  \nabla \log \pi_0, \nabla \pi_t\rangle \\
    & \leq 0 + \frac{1}{2}\mathbb{E}_{\pi_t}(\|\nabla \log \pi_0\|^2_2+ \Delta \log \pi_0) 
   \end{split}
\end{align}
With the assumption that $g$ is $\alpha$-strongly convex and $\beta$-smooth, basing on  Lemma 12 in \cite{vempala2019rapid}, we could obtain
\begin{align}\nonumber
    \begin{split}
         \frac{\partial D_{\mathrm{KL}}(\pi_t|| \pi_0)}{\partial t}  
         & \leq \frac{1}{2}\mathbb{E}_{\pi_t}(\|\nabla \log \pi_0\|^2+ \Delta \log \pi_0)   \\ 
         & \leq \frac{2\beta^2}{\alpha}D_{\mathrm{KL}}(\pi_t|| \pi_0) +d\beta. 
    \end{split}
\end{align}
Solving this differential inequality with the boundary condition $D_{\mathrm{KL}}(\pi_0|| \pi_0) = 0$ yields 
\begin{equation}\nonumber
    D_{\mathrm{KL}}(\pi_t|| \pi_0) \leq \frac{d\alpha}{2\beta}\left[\exp\left(\frac{2\beta^2t}{\alpha}\right)-1\right].
\end{equation}
By Pinsker’s inequality, we finally have
\begin{equation}\nonumber
    {\mathrm{TV}}(\pi_0, \pi_t) \leq \sqrt{\frac{d\alpha}{4\beta}\left[\exp\left(\frac{2\beta^2t}{\alpha}\right)-1\right]}.
\end{equation}
\end{proof}

\subsection{Proof of Proposition \ref{proposition:quadratic} }\label{subappendix:ex4}

\begin{proof}
For  ease of notation, in the proof, for any function $h_t$, we denote its first-order  and second-order time derivative by $\dot h_t$ and $\ddot h_t$, respectively.

For $t\in (0,\eta]$, let 
\begin{align}\nonumber
    \begin{split}
        \pi_t  &= \exp\left[-f(x)-g_t(x)-\phi_t\right], \\
        \pi_0 & = \lim_{t \rightarrow 0}\pi_t = \exp\left[-f(x)-g(x)-\phi_0\right]
    \end{split}
\end{align}
where 
\begin{align}\nonumber
    \begin{split}
    g_t(x) & = -\log\int \exp\left[-g(y)-\frac{1}{2t}\|x-y\|^2\right] \rd y, \\
    \phi_t & =\log\int \exp\left[-f(x)-g_t(x)\right] \rd x, \\
    \phi_0 & =\log\int \exp\left[-f(x)-g(x)\right] \rd x.
    \end{split}
\end{align}
It follows that the time derivatives of their terms are  
\begin{align}\label{eq:eq4}
    \begin{split}
        \dot \phi_t & =  -\mathbb{E}_{\pi_t}(\dot g_t) \\
        \dot \pi_t &= -\pi_t\left[\dot g_t+\dot \phi_t\right]
    \end{split}
\end{align}
And
\begin{equation}\label{eq:eq5}
\dot g_t = -\frac{1}{2t^2}\mathbb{E}_{r_t(.|x)}\left(\|x-y\|^2\right)
\end{equation}
with
\begin{equation}\label{eq:eq6}
r_t(.|x) = \exp\left[-g(y)-\frac{1}{2t}\|x-y\|^2+g_t(x)\right].
\end{equation}
Combining with \eqref{eq:eq4} and \eqref{eq:eq5}, we have
\begin{align}\label{eq:k_t}
    \begin{split}
        K(t):= \frac{\partial D_{\mathrm{KL}}(\pi_0|| \pi_t)}{\partial t} &= \int \pi_0(\dot g_t+\dot \phi_t) \\
        &= \mathbb{E}_{\pi_0}(\dot g_t) -\mathbb{E}_{\pi_t}(\dot g_t) \\
        &= \frac{1}{2t^2}\mathbb{E}_{\pi_t}\left[\mathbb{E}_{r_t(.|x)}\left(\|x-y\|^2\right)\right]- \frac{1}{2t^2}\mathbb{E}_{\pi_0}\left[\mathbb{E}_{r_t(.|x)}\left(\|x-y\|^2\right)\right].
    \end{split}
\end{align} 
To perform a more detailed analysis, we consider the second-order time derivative of $ D_{\mathrm{KL}}(\pi_0|| \pi_t)$, which satisfies
\begin{equation}\nonumber
\dot{K}(t) =  \int \pi_0\ddot g_t- \int \pi_t\ddot g_t -\int \dot \pi_t \dot g_t .
\end{equation}
Plugging \eqref{eq:eq4} into $\int \dot \pi_t \dot g_t$ yields
\begin{equation}\label{eq:eq7}
    \int \dot \pi_t \dot g_t = -\Var_{\pi_t}(\dot g_t ).
\end{equation}
Based on \eqref{eq:eq5} and \eqref{eq:eq6}, we have
\begin{equation}\label{eq:ddot_g_t}
    \ddot g_t = -\frac{1}{4t^4}\Var_{r_t(.|x)}[\|x-y\|^2] + \frac{1}{t^3}\mathbb{E}_{r_t(.|x)}\left[\|x-y\|^2\right].
\end{equation}
Hence, with \eqref{eq:eq5}, \eqref{eq:eq7}  and \eqref{eq:ddot_g_t},
\begin{align}\label{eq:dot_k_t}
    \begin{split}
        \dot{K}(t) =&  \mathbb{E}_{\pi_0}\left[\frac{1}{t^3}\mathbb{E}_{r_t(,|x)}\|x-y\|^2-\frac{1}{4t^4}\Var_{r_t(,|x)}\|x-y\|^2\right] \\
        &- \mathbb{E}_{\pi_t}\left[\frac{1}{t^3}\mathbb{E}_{r_t(,|x)}\|x-y\|^2-\frac{1}{4t^4}\Var_{r_t(,|x)}\|x-y\|^2\right] \\
        &+ \frac{1}{4t^4}\Var_{\pi_t}\left[\mathbb{E}_{r_t(,|x)}\|x-y\|^2\right].
    \end{split}
\end{align}
In what follows, we plug in  $g(y) = \frac{1}{2\sigma^2}\|y-u\|^2_2$. From \eqref{eq:eq6},
\begin{align}
    r_t(.|x) & = \exp\left[-g(y)-\frac{1}{2t}\|x-y\|^2+g_t(x)\right] \nonumber\\
    &= \mathcal{N}\left(t(\sigma^2+t)^{-1}u+\sigma^2(\sigma^2+t)^{-1}x, t\sigma^2(\sigma^2+t)^{-1}\right).
\end{align}\nonumber
Hence, 
\begin{align}
    \begin{split}
    \bbE_{r_t(,|x)}\|x-y\|^2 & = \|t(\sigma^2+t)^{-1}(u-x)\|^2+  t\sigma^2(\sigma^2+t)^{-1}d \\
    \Var_{r_t(,|x)}\|x-y\|^2 &= 4t\sigma^2(\sigma^2+t)^{-1}\|t(\sigma^2+t)^{-1}(u-x)\|^2 + 2t^2\sigma^2(\sigma^2+t)^{-2}d.
    \end{split}
\end{align} 
It follows that
\begin{align}
    \begin{split}
        &\quad \mathbb{E}_{\pi_0}\left[-\frac{1}{4t^4}\Var_{r_t(,|x)}\|x-y\|^2\right] - \mathbb{E}_{\pi_t}\left[-\frac{1}{4t^4}\Var_{r_t(,|x)}\|x-y\|^2\right] \\
        & = -\frac{\sigma^2(\sigma^2+t)^{-1}}{t^3}\left(\bbE_{\pi_0}[\mathbb{E}_{r_t(,|x)}\|x-y\|^2]-\bbE_{\pi_t}[\mathbb{E}_{r_t(,|x)}\|x-y\|^2]\right) \\
        & = \frac{2\sigma^2(\sigma^2+t)^{-1}}{t}K(t),
    \end{split}
\end{align}
and 
\begin{align}
    \begin{split}
       K(t) &= \frac{1}{2t^2}\mathbb{E}_{\pi_t}\left[\mathbb{E}_{r_t(.|x)}\left(\|x-y\|^2\right)\right]- \frac{1}{2t^2}\mathbb{E}_{\pi_0}\left[\mathbb{E}_{r_t(.|x)}\left(\|x-y\|^2\right)\right] \\
        &= \frac{1}{2}\bbE_{\pi_t}[ \|(\sigma^2+t)^{-1}(u-x)\|^2]-\frac{1}{2}\bbE_{\pi_0}[ \|(\sigma^2+t)^{-1}(u-x)\|^2]
    \end{split}
\end{align}
which implies $\lim_{t\rightarrow 0}K(t) = 0$ under mild conditions.
With \eqref{eq:k_t} and \eqref{eq:dot_k_t}, we have
\begin{align}
    \begin{split}
        \dot{K}(t) &\leq  -\frac{2}{\sigma^2+t}K(t) +  \frac{1}{4t^4}\Var_{\pi_t}\left[\|t(\sigma^2+t)^{-1}(u-x)\|^2\right].
    \end{split}
\end{align}
We also assume $f(x)$ is $\alpha_f$-strongly convex, so $\pi_t$ is $\alpha_f+\frac{1}{\sigma^2+t}$-strongly log-concave. Thus, $\pi_t$ satisfies $\alpha_f+\frac{1}{\sigma^2+t}$-Poincar\'e inequality. This follows that
\begin{equation}\nonumber
    \frac{1}{4t^4}\Var_{\pi_t}\left[\|t(\sigma^2+t)^{-1}(u-x)\|^2\right] \le \frac{1}{(\sigma^2+t)^4(\alpha_f+\frac{1}{\sigma^2+t})}\bbE_{\pi_t}\|x-u\|_2^2.
\end{equation}
As we assume the minimizer of $f(x)$ is $u$, the potential of $\pi_t$ is minimized at $u$ as well. By Theorem 1 in \citet{durmus2016high}, we have
\begin{equation}\nonumber
    \bbE_{\pi_t}\|x-u\|_2^2 \le \frac{d}{\alpha_f+\frac{1}{\sigma^2+t}}.
\end{equation}
Combining all the above inequality yields
\begin{equation}\nonumber
    \dot{K}(t) \leq  -\frac{2}{\sigma^2+t}K(t) + \frac{d}{(\sigma^2+t)^4(\alpha_f+\frac{1}{\sigma^2+t})^2} \leq -\frac{2}{\sigma^2+t}K(t) + \frac{d}{(\alpha_f\sigma^4+\sigma^2)^2}
\end{equation}
with the boundary condition $\lim_{t \rightarrow 0}K(t)=0$.
Solving the previous inequalities gives 
\begin{equation}\nonumber
K(t) \leq \frac{dt}{3(\alpha_f\sigma^4+\sigma^2)^2}\left(1+\left(\frac{\sigma^2}{\sigma^2+t}\right)^2+\frac{\sigma^2}{\sigma^2+t}\right) \le\frac{dt}{(\alpha_f\sigma^4+\sigma^2)^2}.
\end{equation}
Thus, we have
\begin{equation}\nonumber
     D_{\mathrm{KL}}(\pi_0|| \pi_{t}) \le \frac{dt^2}{2(\alpha_f\sigma^4+\sigma^2)^2}.
\end{equation}
By Pinsker’s inequality, we  have
\begin{equation}\nonumber
    {\mathrm{TV}}(\pi_0, \pi_t) \leq \frac{t\sqrt{d}}{2(\alpha_f\sigma^4+\sigma^2)}.
\end{equation}
\end{proof}

\end{document}